%% file: unsatcycles.tex
\documentclass[a4paper,11pt]{article}

\usepackage{amsthm,amssymb,amsmath,enumerate, subfig}
\usepackage{graphicx}
\usepackage{boxedminipage}
\usepackage{color}
\allowdisplaybreaks[2]

\def\silent#1\par{\par}

\makeatletter
\renewcommand{\@seccntformat}[1]{\@nameuse{the#1}.\quad}
\makeatother

\let\eps=\varepsilon
\let\epsilon=\varepsilon

\newtheorem{theorem}{Theorem}
\newtheorem{lemma}[theorem]{Lemma}
\newtheorem{fact}[theorem]{Fact}

\newtheorem{definition}[theorem]{Definition}

\newtheorem{corollary}[theorem]{Corollary}

\title{Cycles are strongly Ramsey-unsaturated}
\author{Jozef Skokan and Maya Stein\footnote{Supported by Fondecyt grant 11090141.}}

\date{\today}

\begin{document}

\maketitle

\begin{abstract} 
We call a graph $H$ \emph{Ramsey-unsaturated} if there is an edge in the complement of $H$ such that the Ramsey number $r(H)$ of $H$ does not change upon adding it to $H$. This notion was introduced by Balister, Lehel and Schelp in~\cite{balister}, where it is shown that cycles (except for $C_4$) are Ramsey-unsaturated, and conjectured that, moreover, one may add {\it any} chord without changing the Ramsey number of the cycle $C_n$, unless $n$ is even and adding the chord creates an odd cycle.

 We prove this conjecture for large cycles by showing a~stronger statement: If a graph $H$ is obtained by adding a linear number of chords to a cycle $C_n$, then $r(H)=r(C_n)$, as long as the maximum degree of $H$ is bounded, $H$ is either bipartite (for even $n$) or almost bipartite (for odd $n$), and $n$ is large. 

This motivates us to call cycles {\em strongly} Ramsey-unsaturated.  Our proof uses the regularity method.
\end{abstract}

\thispagestyle{empty}

\section{Introduction}

The Ramsey number $r(H)$ of a graph $H$ is the smallest integer $N$ such that every $2$-colouring of the edges of the complete graph $K_N$ on $N$ vertices contains a monochromatic copy of $H$. 

It has been known \cite{BonErd,FauSch,Ros} for a long time that the cycles $C_n$ on $n$ vertices have Ramsey numbers linear in $n$, while the Ramsey numbers of complete graphs $K_n$ are exponential. So from the beginnings of Ramsey theory one important question has been hovering in the background:  if we keep adding chords to $C_n$ until reaching $K_n$, at which point will the Ramsey number jump up? Very little is known. When we are close to the complete graph, then, for $3\leq n\leq 6,$ it holds that $r(K_n)>r(K_n-e)$, and it is conjectured  (see~\cite{balister}) that this is also true for every $n>6$.

At the other end of the spectrum, that is, when we are close to the cycle, all that is known is the aforementioned result by  Balister, Lehel, and Schelp~\cite{balister}. Using their terms, we call a graph $H$ \emph{Ramsey saturated} if $r(H+e)>r(H)$ for every edge $e$ in the complement of $H$, and \emph{Ramsey unsaturated} otherwise. In~\cite{balister} it is proved that the cycle $C_n$ is Ramsey unsaturated for every $n>4$, and it is conjectured that  one may add any edge $e$ to $C_n$ without increasing the Ramsey number, unless $n$ is even and adding $e$ destroys the bipartiteness of $C_n$ (Conjecture 3 in~\cite{balister}). Balister, Lehel, and Schelp also remark `One would expect to be able to add more than one edge to the cycle without increasing the Ramsey number'.

We show that this is indeed the case, at least asymptotically. Our main result is  that one may add a linear number $cn$ of chords to $C_n$ without changing the Ramsey number, as long as the obtained graph $H$ has bounded  maximum degree $\Delta(H)$, and as long as $H$ is bipartite if $n$ is even, and almost bipartite if $n$ is odd. (A non-bipartite graph is called {\it almost bipartite} if the removal of one vertex results in a bipartite graph.) Since our proof relies on the regularity method, our result holds only for large graphs, with the constant $c$ depending only on $\Delta (H)$.

\begin{theorem}\label{main1}
 For every $\Delta\in\mathbb N$ there are $n_0\in\mathbb N$ and $c>0$ such that, for every $n\geq n_0$ and every collection $D$ of chords of $C_n$ with
 \begin{enumerate}[(1)]
  \item $|D|\leq cn$ and\label{thm1cond1}
  \item $\Delta (C_n\cup D)\leq \Delta$,\label{thm1cond2}
 \end{enumerate}
 the following two conditions are equivalent
 \begin{enumerate}[(a)]
  \item $r(C_n\cup D)=r(C_n)$,
  \item either $n$ is even and $C_n\cup D$ is bipartite, or
        $n$ is odd and $C_n\cup D$ is almost bipartite.
 \end{enumerate}
\end{theorem}

Observe that Theorem~\ref{main1} confirms Conjecture 3 of~\cite{balister}, as an odd cycle plus one chord is almost bipartite.

When $n>4$ is even, $r(C_n)=3n/2-1\leq2(n-1)$ (see \cite{FauSch,Ros}) and the implication $(a)\Rightarrow (b)$ of Theorem~\ref{main1} can easily be seen by considering the following construction: 
Colour all edges of a maximal cut in $K_{2(n-1)}$ blue, and the remaining edges red. By (a), there is a monochromatic copy of $C_n\cup D$. Since the components of the red subgraph of $K_{2(n-1)}$ are too small to contain $C_n\cup D$, there is a blue copy of  $C_n\cup D$ in $K_{2(n-1)}$. Thus $C_n\cup D$ is bipartite.

For odd $n>4$, we know that $r(C_n)=2n-1$ (see \cite{FauSch,Ros}). Take the above colouring of $K_{2(n-1)}$ and add one new vertex to $K_{2(n-1)}$ all whose incident edges are coloured blue. Then as above, we find a blue copy of $C_n\cup D$. This copy cannot be bipartite since it contains the odd cycle $C_n$, hence it must be almost bipartite.

\medskip

The odd case of Theorem~\ref{main1} follows from a slightly more general result. As a generalization of almost-bipartiteness, let us call a graph \emph{$k$-almost bipartite} (for $k\in\mathbb N$) if it contains a set of $k$ independent vertices such that the removal of these vertices leaves a bipartite graph, but the removal of any set of less than $k$ independent vertices does not. Then a graph is $1$-almost bipartite if, and only if, it is almost bipartite.

\begin{theorem}\label{main2}
 For every $\Delta$ and $k\in\mathbb N$  there are $n_0\in\mathbb N$ and $c>0$ such that, for every odd $n\geq n_0$ and every collection $D$ of chords of $C_n$ with
 \begin{enumerate}[(1)]
  \item $|D|\leq c n$ and
  \item $\Delta (C_n\cup D)\leq \Delta$,
 \end{enumerate}
 the following two conditions are equivalent 
 \begin{enumerate}[(a)]
  \item $r(C_n\cup D)=r(C_n)+k-1$,
  \item $C_n\cup D$ is $k$-almost bipartite.
 \end{enumerate}
\end{theorem}

Again, the implication $(a)\Rightarrow (b)$ of Theorem~\ref{main2} is straightforward: Suppose that $r(C_n\cup D)= r(C_n)+k-1=2(n-1)+k$ and $n$ is large (in particular, $n>k$). We partition the vertices of $K_{r(C_n\cup D)}$ into sets $X$, $Y$, and $Z$ such that $|X|=|Y|=n-1$ and $|Z|=k$ and colour all edges within $X$ or $Y$ or $Z$ red and the rest blue. Clearly, the red subgraph has no large enough component to contain $C_n\cup D$, so we find a blue copy of $C_n\cup D$ in $K_{2(n-1)+k}$. Since the blue subgraph is  $k$-almost bipartite, we find that $C_n\cup D$ is $k'$-almost bipartite for some $k'\leq k$. Now, if $k'<k$, then we can use the implication $(b)\Rightarrow (a)$ for $k'$ to obtain a~contradiction. Hence, $C_n\cup D$ is $k$-almost bipartite.

\medskip

Both our results are, apart from the value of $c$, best possible for every sufficiently large $\Delta$. To see this, we first recall the following result of Graham, R\"odl and Ruci\'nski:

\begin{theorem}[Graham, R\"odl and Ruci\'nski, \cite{grr_bip}]\label{thm:grr}
There exists a constant $c_1>1$ such that, for every $\Delta\geq 1$ and for every
$n\geq\Delta+1$ (except for $\Delta=1$ and $n=2,3,5$), there exists a bipartite graph $H$ with $n$ vertices and maximum degree at most $\Delta$ which satisfies $r(H)>c_1^\Delta n$.
\end{theorem}

Let $\Delta\geq 3$ be such that $c_1^{\Delta-2}\geq 4$,  and let $n$ be even and larger than~$2(\Delta+2)$. Then Theorem~\ref{thm:grr} assures that there is a bipartite graph~$H$ on $n/2>5 $ vertices and of maximum degree at most $\Delta-2$ such that $r(H)>c_1^{\Delta-2} (n/2)\geq 2n$. Clearly, we can add $H$ into $C_n$ respecting the natural bipartition of $C_n$. Denote the set of chords by $D$ and note that $\Delta (C_n\cup D)\leq \Delta$ and $|D|\leq e(H)\leq \tfrac12 (\Delta-2) \tfrac{n}{2}<(\Delta/4) n$. On the other hand, we have $$r(C_n)\ =\ 3n/2-1\ <\  2n\ \leq \ r(H)\ \leq\ r(C_n\cup D).$$
Hence, this construction shows that the function from Condition~\eqref{thm1cond1} of Theorem~\ref{main1} cannot be improved to anything above $\Delta n/4$.

A similar construction works for large odd $n$. In this case, let $\Delta\geq 3$ be such that $c_1^{\Delta-2}\geq 6$, $n>\max\{2\Delta, 6k+1\}$,
 and let $H$ be a bipartite graph on $(n+1)/2$ vertices with $\Delta(H)\leq \Delta-2$ and  $r(H)>c_1^{\Delta-2} (n+1)/2\geq 3(n+1)$.

Let $v_1,v_2,\dots, v_n$ be the (consecutive) vertices of $C_n$. We add $k$ chords $v_1v_3$, $v_4v_6,\dots, v_{3k-2}v_{3k}$ and add $H$ into the path $v_{3k+1},v_{3k+2},\dots, v_n$, respecting its natural bipartition. Denote the set of thus created chords by~$D$. Then $C_n\cup D$ had maximum degree $\Delta$, it is $k$-almost bipartite because it contains $k$ vertex-disjoint triangles $v_1v_2v_3$, $v_4v_5v_6,\dots, v_{3k-2}v_{3k-1}v_{3k}$ and removing vertices $v_1, v_4,\dots, v_{3k-2}$ from $C_n\cup D$ yields a bipartite graph, and $|D|\leq e(H)+k\leq \tfrac12(\Delta-2) \tfrac{n}{2}+k< \tfrac{\Delta}{4}n$. 
On the other hand, we have 
$$r(C_n)+k-1\ =\ 2n+k-2\ <\  3(n+1)\ \leq \  r(H)\ \leq \ r(C_n\cup D).$$

\medskip

Furthermore, Theorem \ref{thm:grr} also implies that Theorems \ref{main1} and \ref{main2} cannot be true for non-constant $\Delta$. Suppose that $\Delta= \Delta(n)\to\infty$ as $n\to\infty$ arbitrarily slowly.  Note that we may assume that $2\log n>\Delta$, because for faster growing $\Delta$, or more precisely, for all $\Delta \geq 2\log n$, we can add $K_{\log n, \log n}$ to $C_n$ so that the resulting graph has maximum degree at most $\Delta$ and it is bipartite, but its Ramsey number is at least $r(K_{\log n, \log n})>2n$ for large $n$ (cf.~Chung and Graham~\cite{ChuGra}).

Now, let $n_0$ be such that $c_1^\Delta/\Delta^2 >2$ for every $n>n_0$. Let $H$ be a bipartite graph on $n/\Delta^2>\Delta+1$ vertices and of maximum degree at most~$\Delta$ such that $r(H)>c_1^\Delta (n/\Delta^2)>2n$. Again, we can add $H$ into $C_n$ respecting the natural bipartition of $C_n$. Denote the set of chords by $D$ and note that $|D|\leq e(H)\leq \tfrac12(\Delta \tfrac{n}{\Delta^2})=o(n)$. On the other hand, we have $$r(C_n)\ =\ 3n/2-1\ \leq\ 2n\ <\ r(H)\ \leq \ r(C_n\cup D).$$

\section{Preliminaries}
\subsection{Notation and Two Facts}

We use standard graph-theoretic notation.  For a graph $G$, let $V(G)$ denote the set of its vertices and $E(G)$ the set of its edges. We denote by $|G|$ and $e(G)$ the number of vertices and edges of $G$. Given a set $A\subset V(G)$ of vertices and a set $F\subset E(G)$ of edges, $G[A]$ stands for the subgraph of $G$ induced by the vertices of $A$, and $G-F$ is the subgraph of $G$ with vertex set $V(G)$ and edge set $E(G)\setminus F$.
 
 Given two disjoint sets of vertices, $A$ and $B$, we write $E(A,B)$ for the set of edges with one endpoint in $A$ and the other in $B$, we set $e(A,B) := |E(A,B)|$, and call $G[A,B]$ the bipartite subgraph of $G$ with bipartition $A, B$ and edge set $E(A,B)$. The quantity $d(A,B):=d(A,B)/|A||B|$ is the \textit{density}\/ of the pair $(A,B)$. For a vertex $a$ and a set $B$, we denote by $N(a)$ the set of all vertices adjacent to $a$ and set $\deg(a):=|N(a)|$.

Next, let us clarify some notation on paths: The {\it length} of a path is the number of edges it contains. An {\em odd path} is a path of odd length, an {\em even path} is a path of even length. A path is {\it trivial} if it consists of one vertex only. An {\em $x$--$y$ path} is a path starting at vertex $x$ and ending at vertex $y$. If $x$ and $y$ are vertices of some path $P$, then we denote by $xPy$ the subpath of $P$ starting at $x$ and ending at $y$. If $P$ is a path with one endvertex in a set $A$ and the other one in a set $B$ and does not meet $A\cup B$ otherwise, then we say that $P$ is an {\em $A$--$B$ path}. For a set $A$, a non-trivial $A$--$A$ path is also called an {\em $A$-path}. 

\medskip

We finish this subsection by quickly proving two easy facts, which will be needed in Section~\ref{preparing}.

\begin{fact}\label{bipa}
 Let $G$ be a bipartite graph, and let $H\subseteq G$ be such that all $V(H)$-paths of $G-E(H)$ are odd. Then $H$ has a bipartition $U_1, U_2$ of its vertex set so that each $V(H)$-path in $G-E(H)$ is a $U_1$--$U_2$ path.
\end{fact}
\begin{proof}
 Take any bipartition of $G$ and consider its restriction $U_1, U_2$ to the vertices of $H$. Since the endpoints of every odd path cannot be in the same partite set, the statement of the fact follows.
\end{proof}

 \begin{fact}\label{tripa}
 Let $H', G$ be a tripartite graphs  such that $G$ is obtained from $H'$ by adding at most $|H'|$ pairwise internally vertex disjoint $V(H')$-paths, each of odd length at least $m$, $m\geq 3$. 

  Then there is a graph $H$ with $H'\subseteq H\subseteq G$ and $|H|\leq 3|H'|$ which has a tripartition $U_1, U_2, U_3$ of its vertex set such that each $V(H)$-path in $G-E(H)$ is an odd $U_1$--$U_2$ path of length at least $m-2$.
 \end{fact}
  \begin{proof}
Let $H'$ have colour classes $U'_1$, $U'_2$ and $U'_3$. For each $V(H')$-path $P=abc\ldots xyz$ in $G-E(H')$ that has at least one endvertex ($a$ or $z$) in $U'_3$, we add both $b$ and $y$ to $H'$. More precisely, if $a,z\in U'_3$, then we add $b$ to $U'_1$ and $y$ to $U'_2$, and if only one of $a$, $z$ lies in $U'_3$, say $a\in U'_3$, and $z\in U'_i$, $i\in \{ 1,2\}$, then we add $b$ to $U'_i$ and $y$ to $U'_{3-i}$.
Call the obtained tripartite graph~$H$, with partition classes  $U_1$, $U_2$ and $U_3$. Clearly, the $V(H)$-paths in $G-E(H)$ all have odd length $\geq m-2$, and since $H$ is obtained from $H'$ by adding two vertices from each $V(H')$-path, we also have $|H|\leq 3|H'|$.
\end{proof}
 
\subsection{Structural tools}

We now present some structural results we will use, and which come from~\cite{allen, ErdGall59, cycle-stab}. We also prove a corollary of one of these theorems, which will be tailor-made for our purposes.

The first two are well-known theorems of Erd\H os and Gallai.

\begin{theorem}[Erd\H{o}s, Gallai,~\cite{ErdGall59}]\label{thm:eg}
Every graph $G$ with at least $(n-1)(|G|-1)/2+1$ edges
 contains a cycle of length at least $n\geq 3$.
\end{theorem}
\begin{theorem}[Erd\H{o}s, Gallai,~\cite{ErdGall59}]\label{thm:eg2}
 Let $G$ be a graph with $v(G)\geq 3$ and minimum degree at least $(|G|+1) /2$. Then, for every two vertices $u$ and~$v$, there exists an $u$--$v$-path containing all vertices of $G$.
\end{theorem}
  
 The next result is a recent one on embedding bounded degree graphs.

\begin{lemma}[Allen, Brightwell, Skokan, \cite{allen}]\label{GenEmbed} 
 Given a natural number $\Delta \ge 1$ and any $0<\eps<1/(\Delta^2+4)$, 
 let $F$ be an $n$-vertex graph in which every vertex has degree at 
 least $3\Delta\eps n$, and all but at most $\eps n$ vertices have 
 degree at least $(1-2\eps)n$. Let $J$ be any $n$-vertex graph with 
 $\Delta(J)\leq\Delta$. Then $J\subset F$.
\end{lemma}

We also need a well-known result of Graham, R\"odl and Ruci\'nski on linear bounds for Ramsey numbers of bounded degree graphs.
 
\begin{theorem}[Graham, R\"odl, Ruci\'nski,~\cite{grr}]\label{linram}
There exists a constant $c_2>1$ such that for all $\Delta\geq 2$ and all
$n\geq\Delta+1$, every graph $H$ with $n$ vertices and maximum degree at most $\Delta$ satisfies $r(H)\leq 2^{c_2\Delta(\log\Delta)^2} n$.
\end{theorem}

Finally, we will need stability results on monochromatic cycles in graphs whose edges are multicoloured. (A multi-colouring of the edges of some graphs is an assigment of colours to the edges in which an edge can have more than one colour.) It will come as no surprise that there are two cases, according to the parity of the cycle. We start with the odd case:

\begin{theorem}[Kohayakawa, Simonovits, Skokan, \cite{cycle-stab}]\label{jmy-odd}
 For $\gamma\in (0,10^{-2})$, let $\bar m,\bar M$ be integers with $\bar m> \gamma^{-6}$ and $\bar M \geq (3/2+14\gamma)\bar m$. Suppose that $G$ is a graph on $\bar M$ vertices with $\delta(G)\geq \bar M-\gamma \bar m$. 
 
 Then every multi-colouring of the edges of~$G$ with $2$ colours contains either a monochromatic odd cycle $C_\ell$ for some $\ell \geq (1+\gamma)\bar m$ or 
 there is a vertex $x\in V(G)$, a partition of $V(G)\setminus\{v\}$ into two sets $V_1$ and $V_2$, and an $i\in\{1,2\}$  such that 
 \begin{enumerate}[(a)]
\item all edges of $G[V_1] \cup G[V_2]$ have only colour $i$, 
\item all edges of $G[V_1, V_2]$ have only colour $2-i$, and 
\item $|V_1|,|V_2| < (1+\gamma)\bar m$.\label{sizeV1V2}
 \end{enumerate}
 \end{theorem}

A similar result holds for even cycles.

\begin{theorem}[Kohayakawa, Simonovits, Skokan, \cite{cycle-stab}]\label{jmy-even}
 \hskip-3pt For each $\beta\hskip-2.5pt\in\hskip-2.5pt(0,10^{-6})$, there is an $n_{\rm even}\in\mathbb N$ such that the following holds for all integers $m, M$ with $m> n_{\rm even}$ and $M \geq (3/2-\beta)m$.\\
If the edges of a graph $G$ on $M$ vertices with $\delta(G)\geq M-\beta m$ are multi-coloured with $2$-colours, then $G$ either contains a monochromatic even cycle $C_\ell$ for some $\ell \geq (1+\beta)m$ or there is a partition of  $V(G)$ into sets $V_1$, $V_2$ and $W$, and an $i\in\{1,2\}$ such that 
 \begin{enumerate}[(a)]
\item all but at most $\beta^{1/5}|V_1|^2$ edges of $G[V_1]$ have only colour $i$,\label{thmevena}
\item all but at most $\beta^{1/5}|V_1||V_2|$  edges of $G[V_1, V_2]$ have only  colour $3-i$, and\label{thmevenb}
\item $|V_1|\geq (1-\beta)m$, $|V_2|\geq (1-\beta)m/2$, $|W|<\beta^{1/5}m$.\label{thmevenc}
 \end{enumerate}
\end{theorem}  

In addition to Theorem~\ref{jmy-even}, we will use its following corollary:

\begin{corollary}\label{coro-even}
The sets $V_1$ and $V_2$ from Theorem~\ref{jmy-even} contain sets $A$ and $B$, respectively, such that
 \begin{enumerate}[(a')]
 \item the minimum degree in colour $i$ in $G[A]$ is at least $(1-\beta^{1/20})|A|$,  the maximum degree in colour $3-i$ in $G[A]$ is at most $\beta^{1/20}|A|$,
\item in colour $3-i$ in $G[A,B]$, every vertex of $A$ has degree at least $(1-\beta^{1/20})|B|$ and every vertex of $B$ has degree at least $(1-\beta^{1/20})|A|$, and
\item $|A|\geq (\frac 23-\beta^{1/20})\frac 32 m$, $|B|\geq (\frac 13-\beta^{1/20})\frac 32 m$.
  \end{enumerate}
\end{corollary}

\begin{proof}
Observe that, firstly, condition
(a) implies that, in $V_1$, there are at most $2\beta^{1/10}|V_1|$ vertices that, in colour $3-i$, have degree at least $\beta^{1/10}|V_1|$ inside $V_1$. Secondly,
(b) implies that $V_1$ contains at most $\beta^{1/10}|V_1|$  vertices sending at least $\beta^{1/10}|V_2|$ edges   of colour $i$ to $V_2$. Finally, 
(b) implies that, in $V_2$, there are at most $\beta^{1/10}|V_2|$ vertices that send at least $\beta^{1/10}|V_1|$ edges of colour $i$ to $V_1$. 

Remove these in total at most $3\beta^{1/10}|V_1|$ vertices from $V_1$ and at most $\beta^{1/10}|V_2|$ vertices from $V_2$ to obtain sets $A$ and $B$ as required. In fact, by construction, each vertex of $A$ is incident with at most $\beta^{1/10}|V_1|\leq \beta^{1/20}|A| $ edges of colour $3-i$ that go to other vertices of $A$. Thus (a') holds. Similarly we can see that (b') holds.

For (c'), we have to see that  $|A|\geq (1-\frac 32\beta^{1/20}) m$ and $|B|\geq (1-3\beta^{1/20})\frac m2$. This is true as $|A|\geq (1-3\beta^{1/10})|V_1|$ and $|B|\geq (1-\beta^{1/10})|V_2|$ by construction. Using (c), we thus obtain (c').
\end{proof}

 \subsection{Regularity}
 
 In this section we introduce some very well-known concepts; the reader familiar with regularity is invited to skip this section.

We start with giving the standard definition of regularity.

\begin{definition}
Given $\epsilon>0$ and disjoint subsets $A, B$ of the vertex set of a graph $G$, we say that the pair $(A, B)$ is $\epsilon$\textit{-regular in} $G$  if, for every pair $(A', B')$ with $A'\subseteq A$, $|A'|\geq \epsilon |A|$, $B' \subseteq B$, $|B'|\geq \epsilon |B|$, we have
$$ 
 \left| d(A',B')-d(A,B)\right| <\epsilon.
$$
When there is no danger of confusion, we simply say that the pair $(A,B)$ is $\eps$-regular.
\end{definition}

We now turn to the regularity lemma.  
 
\begin{theorem}[Regularity Lemma; Szemer\'edi, \cite{szemeredi78}]\label{regu}
 For every $\epsilon>0$ and $m_{\rm reg}>0$, there are $M_{\rm reg}$ and $n_{\rm reg}$ such that, for each graph $G$ on at least $ n_{\rm reg}$ vertices, there exists a partition $V_0, V_1,\dots, V_t$ of $V(G)$ such that
\begin{itemize}
 \item [(i)] $m_{\rm reg}\leq t \leq M_{\rm reg}$;
 \item [(ii)]$|V_1|=|V_2|=...=|V_t|$ and $|V_0|\leq \epsilon |V(G)|$;   
 \item [(iii)] for each $i$, $1\leq i\leq t$, all but at most $\epsilon t$ of the pairs $(V_i, V_j)$, $1\leq j\leq t$, are  $\epsilon$-regular. 
\end{itemize}
We shall refer to a partition of $V(G)$ satisfying coditions (i)--(iii) as $\epsilon$-regular with respect to $G$.
\end{theorem}
 
 We continue by stating two facts on $\eps$-regular pairs, which are very well-known, and can be found for example in~\cite{KomSim}.
 
\begin{fact}\label{fact1}
 Let $(V_1, V_2)$ be an $\epsilon$-regular pair of density  $d\geq 2\epsilon$ in some graph $G$. Then all but at most $\epsilon|V_1|$ vertices $v_1\in V_1$ are such that $\deg(v_1)\geq (d-\epsilon) |V_2|$. We shall refer to such a vertex $v_1$ as $\epsilon$-typical in $G$ with respect to $V_2$.
\end{fact}

\begin{fact}\label{fact2}
 Let $0<\epsilon\leq 1/4$  and let $(V_1, V_2)$ be an $\epsilon$-regular pair of density~$d$. Then, for every $V'_1\subseteq V_1$, $V'_2 \subseteq V_2$ such that $|V'_1|\geq \epsilon^{1/2} |V_1|$ and $|V'_2|\geq \epsilon^{1/2}  |V_2|$, the pair $(V'_1, V'_2)$ is $\epsilon^{1/2}$-regular with density at least $d-\epsilon$.
\end{fact}

 We now state the so-called embedding lemma. Given a graph $R$ and a~positive integer $s$, let $R_s$ be the graph obtained from $R$ by replacing each vertex $v$ of $R$ with an independent set $V_v$ of size $s$, and each edge $vw$ of $R$ with the complete bipartite graph on $V_v$ and $V_w$.
 
\begin{lemma}[Embedding lemma, \cite{KomSim}]\label{blow-up}
  For all $\Delta\in\mathbb N$ and $p>0$, if $\eps'>0$ is such that $(p-\eps')^\Delta-\Delta\eps'>p^\Delta/2$, then the following holds for every graph $H$ of maximum degree $\Delta$.

Let $R$ be a graph on vertices $1,\ldots ,r$, and let $G$ be a graph on the union of the sets $V_1,\dots,V_r$, where each $V_i$ has the same size $\lambda$, and for each edge $ij\in E(R)$ the pair $(V_i,V_j)$ is $\eps'$-regular and of density at least $p$. Suppose $s\leq p^\Delta \lambda/2$. If $H\subseteq R_s$, then also $H\subseteq G$.
\end{lemma}
 In particular, we note that the conditions of Lemma \ref{blow-up} are satisfied if
\begin{equation}\label{blw-up:1}
 \eps'\leq \frac{p^\Delta}{4\Delta} \quad \mbox{ and } \quad s\leq 2\Delta\eps' \lambda.
\end{equation} 
  
  \subsection{Embedding long paths}
  
In this section we prove an important lemma, Lemma~\ref{lem:path}, which shall be used later in the proof of our main theorems. This lemma is about embedding a~family of long paths into a sequence of $\eps$-regular pairs. As a~starting point, we use a lemma from~\cite{Ben2010} that shows under which condition we can embed a~path into one regular pair.

\begin{lemma}[Benevides, \cite{Ben2010}]\label{benevides}
  For every $\delta'\in (0,1)$ and for every $\epsilon'\in (0,\frac{\delta'}{20})$, there is an $n_{\ref{benevides}}=n_{\ref{benevides}}(\delta', \eps')$ with the following property: Let $n\geq n_{\ref{benevides}}$ and let $(X_1, X_2)$ be an $\epsilon'$-regular pair (in some graph) with density at least $\delta'$ and with $|X_1|=|X_2|=n$. Then, for every $\ell$, $1\leq\ell\leq n-2\epsilon' n/\delta$, and for every two vertices $v_1\in X_1, v_2\in X_2$ such that $\deg(v_1), \deg(v_2)\geq \delta' n/2$, there exists a $v_1$--$v_2$-path of length $2\ell+1$.
\end{lemma}
 
We now show how to embed more paths into a longer sequence of regular pairs.
 
 \begin{lemma}\label{lem:path}
Let $G$ be a graph, let $\eps\in (0, 1/625)$, and  let $P_1,\dots,P_k$ be a collection of pairwise internally vertex disjoint paths of odd lengths. Suppose that $\ell\in{\mathbb N}$ is odd and $V_1,\ldots, V_\ell\subseteq V(G)$ are such that
  \begin{enumerate}[(i)]
  \item $ |V_1|=|V_2|=\ldots =|V_\ell|\geq n_{\ref{benevides}}(\eps^{1/4}, \eps^{1/2})/4\eps^{1/2}$;\label{lem:pathe}
  \item $V_1,\ldots, V_\ell$ are pairwise disjoint, except (possibly) for $V_1$ and $V_\ell$;\label{lem:pathh} 
  \item for all $i\in[\ell-1]$, the pair $(V_i,V_{i+1})$ is $\eps$-regular of density at least~$2\eps^{1/4}$;\label{lem:pathf}
  \item $k\leq \eps^{1/2}|V_1|$;\label{lem:pathg}
  \item  for all $i\in [k]$,  $p^i:=|E(P_i)| \geq 3\ell$;\label{lem:patha}
  \item $\displaystyle\sum_{i=1}^kp^i\leq  \big(1-2\epsilon^{1/4}\big)(\ell-2)|V_1|$; \label{lem:pathb}
  \item for all $i\in [k]$, the first vertex of $P_i$ is embedded in a vertex $a_i\in V_1$ and the last vertex of $P_i$ is embedded in a vertex $b_i\in V_\ell$; and\label{lem:pathc}
  \item for all $i\in [k]$,  $a_i$ is $\eps$-typical in $G$ with respect to $V_2$ and $b_i$ is $\eps$-typical in $G$ with respect to $V_{\ell -1}$.\label{lem:pathd}
   \end{enumerate}
Then we can extend $A:=\bigcup_{i=1}^k\{a_i,b_i\}$ to an embedding of all of $\bigcup_{i=1}^k P_i$ in~$G$ such that the only vertices from $V_1\cup V_\ell$ used are those from $A$.
 \end{lemma}
Notice that the sets $V_1$ and $V_\ell$ can have a non-empty intersection. Consequently, in such a case, it is possible to choose $a_i=b_i$ and obtain a cycle of length $p^i$ instead of path of length $p^i$. However, in this paper, we will always make sure that $a_i\neq b_i$.
\begin{proof}
For every $j\in[\ell-1]$ and for all $i\in[k]$, we choose an odd number~$q^i_{j}$ that will indicate how many edges of $P_i$ we plan to embed into the pair $(V_j,V_{j+1})$. We choose the numbers $q^i_{j}$ so that they satisfy the following three conditions:
 \begin{enumerate}[(a)]
  \item $q^i_{j}=1$ if $j$ is odd, and $q^i_{j}\geq 3$ if $j$ is even; \label{even1}
  \item $\displaystyle\sum_{j=1}^{\ell-1} q^i_{j}=p^i$ for each $i\in[k]$; and \label{even2}
  \item $\displaystyle \sum_{i=1}^k (q^i_{j}+1)\leq 2 \big(1-2\epsilon^{1/4}\big)|V_j|$ for each even $j\in[\ell -1]$.\label{fitsinthepair}
 \end{enumerate}
  
 Observe that such a choice is possible because of  asumptions~\eqref{lem:patha} and~\eqref{lem:pathb}.   
 \medskip
 
 \begin{figure}[ht]
  \centering
{\label{fig:pathlemma}
\def\svgwidth{320pt}
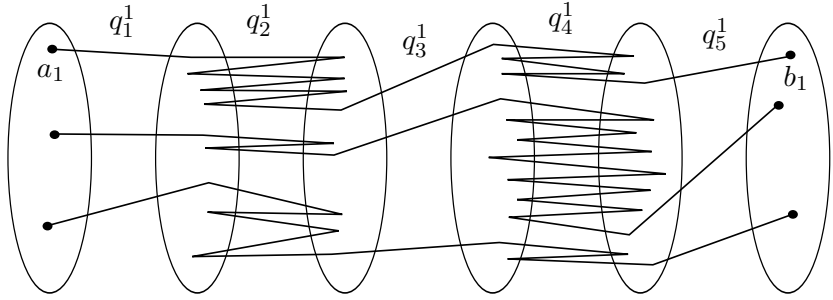}                
  \caption{The numbers $q_j^i$ stand for the number of edges of $P_i$ we wish to embed in the pair $(V_j,V_{j+1})$.}
  
\end{figure}
 \medskip

 Now, for $j\in[\ell]$ and $i\in[k]$, set 
 $$
  Q^i_j:=\sum_{h=1}^{j-1} q^i_{h}+1
 $$ 
 and let $v^i_j$ denote the $Q^i_j$-th vertex on $P_i$. In particular, the preimage of $w^i_1:=a_i$ is  $v^i_1$ and  the preimage of $w^i_{\ell}:=b_i$ is  $v^i_{\ell}$. 
  
 Next, successively, for each $i\in[k]$, we embed $P_i$ in the following way: For each $j= 2,\ldots \ell -1$,  we will consider certain sets $W^i_j\subset V_j$. The idea is that, for even $j$, the sets $W^i_j$ and $W^i_{j+1}$ contain possible images for $V(v^i_{j}P_iv^i_{j+1})\cap V_j$. The sets $W^i_j$ will satisfy the following conditions. 
 \begin{enumerate}[(a)]
\addtocounter{enumi}{3}
  \item $W_2^i\subset N(w^i_1)\cap V_2$, $W_{\ell-1}^i\subset N(w^i_\ell)\cap V_{\ell -1}$, and $W_j^i\subset  V_j$ for each $j=3,\dots,\ell-2$;
  \item $W^i_{j}$ contains no vertex of $P_1,\dots, P_{i-1}$, for  any $j=2,\dots, \ell -1$; 
  \item $\big(1-2\epsilon^{1/4}\big)|W^i_{j}|>(q^i_{2\lfloor \frac j2\rfloor}+1)/2$ for each $j=2,\dots, \ell -1$; 
  \item $|W^i_{j}|>4\epsilon^{1/2}|V_j|$ for $j=2,\dots, \ell -1$; and
  \item $|W^i_j| = |W^i_{j+1}|$ for each even $j=2,4,\dots,\ell-2$.\label{gleichgross}
 \end{enumerate}

\smallskip

Observe that (g), together with Fact \ref{fact2}, ensures that, for $j=2,\dots, \ell -2$,
\begin{equation}\label{Wijarenice}
\text{  $(W^i_j,W^i_{j+1})$ is $\eps^{1/2}$-regular with density at least $2\eps^{1/4}-\eps > \eps^{1/4}$.}
\end{equation}

First, assuming the existence of the sets $W^i_j$, let us embed all the $P_i$. We will show later that the sets $W^i_j\subset V_j$ satisfying (d)--(h) do exist. We shall start with the vertices $v^i_j$, which will be placed in the following way.

For $j=2,\ldots \ell-1$, the vertex $v^i_{j}$ is embedded in a vertex $w^i_{j}$ of $W^i_{j}$ such that 
\begin{subequations}
\begin{align}
&\text{ $w^i_{j}$ is $\epsilon^{1/2}$-typical in $G$ with respect to $W^i_{j-1}$ and $W^i_{j+1}$,}\label{Wija}\\ 
\intertext{and}
&\text{$w^i_jw^i_{j+1}$ is an edge of $G^r$ if $q^i_j=1$.}\label{Wijb}
\end{align}
\end{subequations}

By Fact \ref{fact1}, all but at most $2\epsilon^{1/2} |W^i_{j}|$ vertices in $W^i_{j}$ satisfy \eqref{Wija}.  In order to see \eqref{Wijb}, observe that the sets of vertices satisfying \eqref{Wija} for $j$ and for $j+1$ are large enough to apply $\epsilon^{1/2}$-regularity of $(W^i_j,W^i_{j+1})$, hence there must be an edge $w^i_jw^i_{j+1}$  between them. (Note that condition \eqref{Wijb} holds also for $j=1,\ell-1,$ because of (d).)

The last step of our embedding is the embedding of the subpaths $v^i_{j}P_i v^i_{j+1}$ for each even $j\in[\ell]$.   By~\eqref{even1} we know that $v^i_{j}P_i v^i_{j+1}$ has odd length $q^i_j\geq 3$. We wish to embed the edges of $v^i_{j}P_i v^i_{j+1}$ into the pair $(W^i_j,W^i_{j+1})$. For this,  we plan to use Lemma~\ref{benevides} with $\delta'=\eps^{1/4}$ and $\eps'=\eps^{1/2}$. The assumptions of Lemma~\ref{benevides} are verified in~\eqref{Wijarenice}, (f), (h) and \eqref{Wija}, and the only condition being left to check is that $|W_j^i|\geq n_{\ref{benevides}}$. But this follows from
\[
 |W_j^i|\overset{(g)}\geq 4\epsilon^{1/2}|V_j|
            \overset{\eqref{lem:pathe}}
            \geq n_{\ref{benevides}}(\eps^{1/4}, \eps^{1/2})
            = n_{\ref{benevides}}(\delta', \eps').
\]

This completes the embedding of the paths $P_i$. It only remains to show that the sets $W^i_j\subset V_j$ satisfying (d)--(h) do exist. Let $i\leq k$ be given and assume we have already embedded all paths $P_h$ with $h<i$. Let $U$ be the set of all vertices used by our embedding of $\displaystyle\bigcup_{h<i} P_h$.

Now, for $j=2, \ell-1$, choose $W_j^i\subseteq  N(w^i_j)\cap V_j-U$  and, for  $j=3,\dots,\ell-2$, $W_j^i\subseteq V_j-U$ as large as possible, but such that $|W^i_j| = |W^i_{j+1}|$ for each even $j=2,4,\dots,\ell-2$. This choice clearly guarantees conditions (d), (e), and~(h).  
 
To show (g), observe that,  for $j=2, \ell-1$, we have
\begin{align*}
 |N(w^i_j)\cap V_j-U| &\overset{\eqref{lem:pathd}}\geq 
    \lfloor (2\eps^{1/4}-\eps)|V_j| \rfloor - k\\
 & \overset{\eqref{lem:pathg}}> \eps^{1/4}|V_j| - \eps^{1/2}|V_j|\\
 &\ >  4\eps^{1/2}|V_j|.
 \end{align*}
 Furthermore, for $j=3\ldots, \ell-2$, we know by~\eqref{fitsinthepair} that
 \begin{equation*}
 |V_j-U|  \geq |V_j| - (1-2\eps^{1/4})|V_j| \geq 4\eps^{1/2}|V_j|,
  \end{equation*}
 thus (g) is established.
 
 In a similar way we see (f). For all even $j\in[\ell -1]$, we have
\begin{align*}
 |V_j-U| & = |V_j| - \sum_{h=1}^{i-1} \frac{q^h_j+1}{2} \\
 &\overset{\eqref{fitsinthepair}}\geq \frac{1}{1-2\eps^{1/4}}\sum_{h=1}^{k} \frac{q^h_j+1}{2} - \sum_{h\in[k], h\neq i} \frac{q^h_j+1}{2}\\
  &\geq \frac{1}{1-2\eps^{1/4}} \frac{q^i_j+1}{2}.
  \end{align*}
  For odd $j\in[\ell -1]$ it suffices to notice that $|W^i_j|=|W^i_{j-1}|$ by~\eqref{gleichgross}.

 \end{proof}

\section{The proof of the main theorems}

\subsection{Fixing the parameters}

We shall prove at the same time, and with the same parameters $c$ and $n_0$, Theorem~\ref{main1} for even $n\geq n_0$ and Theorem~\ref{main2} for odd $n\geq n_0$. Recall that we only need to worry about the implication (b) $\Rightarrow$ (a).

Suppose that $\Delta>2$ and, for odd $n$,  also a fixed $k\in \mathbb N$ are given. Set 
\begin{equation}\label{epsdef}
 \eps:=\frac{1}{200k\cdot \Delta^{4000 c_2\Delta(\log \Delta)^2}}
\end{equation}
where the constant $c_2$ comes from Theorem~\ref{linram}, and set
\begin{equation}\label{defmreg}
  m_{\rm reg}:= \frac{n_{\rm even}}{\eps^6},
\end{equation}
where $c_2>1$ comes from Theorem \ref{linram} and $n_{\rm even}$ is given by Theorem~\ref{jmy-even} for the input 
\[
\beta:=\epsilon^{1/5}.
\]
For the parameters $\epsilon$ and $m_{\rm reg}$ the Regularity Lemma (Theorem~\ref{regu}) yields numbers $M_{\rm reg}$ and $n_{\rm reg}$.
We now set 
\begin{equation}\label{c}
 c:=\frac{\eps^2}{16M^2_{\rm reg}}
\end{equation}
and 
\begin{equation}\label{n0}
 n_0:=\frac{M_{\rm reg}n_{\ref{benevides}}}{4\eps^{1/2}(1-\eps)},
\end{equation}
where $n_{\ref{benevides}}$ is given by Lemma \ref{benevides} for $\delta':=\eps^{1/4}$ and $\eps':=\eps^{1/2}$.

Now suppose $n\geq n_0$ and let
\[
N := \left\{ 
\begin{array}{l l}
  r(C_n) & \quad \mbox{if $n$ is even,}\\
  r(C_n) + k -1 & \quad \mbox{if $n$ is odd.}\\ \end{array} \right. 
  \]
Let a $2$-colouring of the edges of $K_N$ be given, with colours red and blue, say, which induce the (spanning) subgraphs $G^r$ and $G^b$. 

Further, let $D$ be a set of chords of $C_n$ satisfying assumptions~\ref{thm1cond1} and~\ref{thm1cond2} of Theorems~\ref{main1} and~\ref{main2}, and such that $C_n\cup D$ is bipartite if $n$ is even and $k$-almost bipartite if $n$ is odd. Our aim is to find a monochromatic copy of $C_n\cup D$ in~$K_N$.

\subsection{Applying regularity to $G^r$ and $G^b$}

The regularity lemma (Theorem~\ref{regu}) applied to $G^r$ yields a partition $V_0$, $V_1,\ldots , V_t$ of $V(G^r)$ that is $\eps$-regular with respect to $G^r$. It is well-known that this partition is also $\eps$-regular with respect to~$G^b$. 

Substituting each $V_i$, $i>0$, with a vertex $i$ and each $\eps$-regular pair $(V_i, V_j)$ with an edge $ij$, we obtain the so-called reduced graph $R$ with vertex set $[t]=\{1, 2,\dots, t\}$ such that
\begin{enumerate}[(I)]
 \item $M_{\rm reg}\geq t\geq m_{\rm reg}$;\label{Rlarge}
 \item $N/t\geq |V_i|\geq (1-\eps)N/t$;\label{sizeVi}
 \item $\delta(R)\geq t-1-\eps t$.\label{delta}
\end{enumerate}

We define an edge-(multi)colouring of $R$ as follows. Colour the edge $ij\in E(R)$  red (blue) if the density of $(V_i, V_j)$ in $G^r$ (in $G^b$) is at least $d$, where
\begin{equation}\label{defd}
 d:=4\max\left\{\eps^{1/100}, (4\Delta\cdot 2\eps)^{1/\Delta} \right\}<\frac{1}{\Delta^{40}}.
\end{equation}
Notice that every edge of $R$ receives a colouring because $d<1/2$.
 
Let us fix some more notation: if $xy$ is a red/blue edge in $R$ (or in $G^r$ or in $G^b$), then well call $x$ a {\em red/blue neighbour} of $y$.
 
 \subsection{Preparing $C_n\cup D$}\label{preparing}
 
  Set 
  \[
  z:= \frac{\eps^2 n}{M_{\rm reg} |D|}
  \]
 and, since $|D|\leq cn$, notice that
\begin{equation}\label{z_lb}
 z\geq \frac{\eps^2}{cM_{\rm reg}} \overset{\eqref{c}}\geq 16M_{\rm reg}.
\end{equation}
    Let $H''$ be the subgraph of $C_n\cup D$ induced by $D$ and all $V(D)$-paths in $C_n$ of length at most $z$. Then 
 \begin{equation}\label{sizeH''}
   |H''| \leq 2z|D| = 2 \frac{\eps^2 n}{M_{\rm reg} }.
 \end{equation}
 Now, $C_n\cup D- E(H'')$ is the union $\mathcal P''$ of several  $V(D)$-paths, each of them longer than $z$. We wish to have them all of the same parity. For this reason, for each such even path $P=v_1,v_2,\dots, v_{2i+1}$, we remove vertex $v_2$ and edge $v_1v_2$ from $P$, add them to $H''$, and call the obtained graph $H'$. Observe that we have added at most $|H''|$ new vertices. Thus,
 \begin{equation}\label{sizeH'}
     |H'| \leq 2|H''|,
  \end{equation}
  and
  \begin{equation}\label{oddpaths}
  \text{all $V(D)$-paths of $C_n\cup D- E(H')$ are odd.}
  \end{equation}
  
  If $n$ is even, then $C_n\cup D$, and thus also $H'$, is bipartite. By Fact~\ref{bipa} and by~\eqref{oddpaths}, we can arrange the colour classes $U_1$, $U_2$ of $H'$ so that $C_n\cup D- E(H')$ is the edge-disjoint union of  $U_1$--$U_2$-paths of length at least $z-1$. Set $H:=H'$.
    
If $n$ is odd, then, since $C_n\cup D\supseteq H'$ is $k$-almost-bipartite, the graph $H'$ is $3$-colourable. By Fact~\ref{tripa}, there is a supergraph $H$ of $H'$ contained in $C_n\cup D$, with tripartition $U_1, U_2, U_3$ of its vertex set such that all $V(H)$-paths in $C_n\cup D-E(H)$ are odd $U_1$--$U_2$ paths of length at least $z-3$.
  
  So, in both cases we have obtained a  graph $H$ with partition classes $U_1$, $U_2$ and (if $n$ is odd) $U_3$ such that, by~\eqref{sizeH''} and by~\eqref{sizeH'},
 \begin{equation}     \label{sizeH} 
     |H| \leq 3|H'| \leq 12 \frac{\eps^2 n}{M_{\rm reg} }.
 \end{equation}    
Furthermore, for each $V(H)$-path $P$ of $C_n-E(H)$, we have that
  \begin{equation}
     \text{ $P$ is an odd $U_1$--$U_2$ path,}\label{U1U2paths} 
  \end{equation} 
 and, since $z\geq 16M_{\rm reg}\geq 6$ by~\eqref{z_lb}, also
 \begin{equation}  
   \text{$|E(P)|\geq z-3\geq \frac z2 \geq 8M_{\rm reg}$.}\label{howlong}
 \end{equation}
   
 From now on, we split our proof into two cases, depending on the parity of~$n$. We shall deal with even $n$ in Section~\ref{sec:even}, and with odd $n$ in Section~\ref{sec:odd}.
 
  \subsection{Even $n$}\label{sec:even}
 
 We wish to apply Theorem~\ref{jmy-even} to $G:=R$ with $\beta=\eps^{1/5}$ (as fixed above),  $m:=\lceil2t/3\rceil$, and $M:=t$. Clearly,  by~\eqref{Rlarge} and by the choice of $m_{\rm reg}$, we have that $m\geq \lceil\frac 23 m_{\rm reg}\rceil>n_{\rm even}$ and 
 \begin{equation}\label{tme}
   t\geq m_{\rm reg}>1/\eps.
 \end{equation}
 Thus, as $m\leq \frac 23 t+1=\frac 23 M+1$, we obtain that
 \[
 M\geq \frac 32m-\frac 32\overset{\eqref{tme}}\geq \frac 32m-\frac 32\eps M\geq \frac 32m-\frac 94\eps m \geq \left(\frac 32-\beta\right)m.
 \]
 Finally, since $4\eps\leq\beta$, we know that 
 \begin{align}\label{mindegR}
  \delta(R)& \overset{\eqref{delta}}\geq M-1-\eps t\notag\\
 & \overset{\eqref{tme}}> M- 2\eps t  \notag\\
 & \geq \  M- 4\eps \frac t2\notag\\
 & \geq\  M-4\eps m\notag \\
 & \geq\  M-\beta m.
\end{align}

 So, Theorem~\ref{jmy-even} and Corollary~\ref{coro-even} yield either 
 \begin{enumerate}[\bf (A)]
\item \label{evencycle} $R$ has a monochromatic even cycle $C_\ell$  with $\ell\geq \frac 23(1+\eps^{1/5})t$, or
 \item \label{evenstructure} there is a colour $i\in\{1,2\}$, and sets $A,B\subseteq V(R)$ so that, for $\xi=\eps^{1/100}<\tfrac{1}{16^{10}}$, 
  \begin{enumerate}[(B1)]
 \item the minimum degree in colour $i$ in $R[A]$ is at least $(1-\xi)|A|$ and the maximum degree in colour $3-i$ in $R[A]$ is at most $\xi|A|$,
 \item in colour $3-i$ in $R[A,B]$, every vertex in $A$ has degree at least $(1-\xi)|B|$ and every vertex in $B$ has degree at least $(1-\xi)|A|$,
 \item $|A|\geq (\frac 23-\xi)t$ and $|B|\geq (\frac 13-\xi)t$.
 \end{enumerate}
  \end{enumerate}

  Depending on which case we obtained, we split our proof further into two cases.
  
 \subsubsection{Case~\eqref{evencycle}}\label{sec:evencycle}
  
  In this case, $R$ contains an even monochromatic (say, red) cycle $C_\ell$ with vertices $V_1, V_2, \ldots, V_\ell, V_1$ (in this order). Let $V_1'\subseteq V_1$, $V_\ell'\subseteq V_\ell$ contain all those vertices that are $\eps$-typical in $G^r$ with respect to $V_2$ or $V_{\ell-1}$, respectively. Thus, because of Fact~\ref{fact1}, we have  
$$
  \lambda:=\min\{|V'_1|,|V'_\ell|\} \geq (1-\eps)|V_1|.
$$
It follows from the definition of $\eps$-regularity that the pair $(V'_1, V'_\ell)$ is $2\eps$-regular with density at least $d-\epsilon>d/2\geq (4\Delta\cdot 2\eps)^{1/\Delta}$.  Moreover, by~condition 2 of Theorems~\ref{main1} and~\ref{main2}, we know that $\Delta (H)\leq\Delta$.
  
  Thus we verified almost all conditions of Lemma~\ref{blow-up}, which we would like to use with parameters $\Delta$, $p:=d-\eps$, $\eps':=2\eps$, $\lambda$, $s:=|H|$ and graphs $R:=K_2$, $H$, $G:=G^r[V'_1\cup V'_\ell]$ in order to embed the bipartite graph $H$ into $G^r[V'_1\cup V'_\ell]$. By \eqref{blw-up:1}, it is only left  to check that
  \[
  |H|\leq  2\Delta\cdot 2\eps\cdot\lambda.
  \]
  
This inequality holds because we have
 \begin{align}\label{Hsmallenough}
  |H| & \overset{\eqref{sizeH}}\leq 12\frac{\eps^2}{M_{\rm reg}}\cdot n
  \leq 12\frac{\eps^2}{M_{\rm reg}}\cdot N\notag\\
 & \overset{\eqref{epsdef}}\leq 4\Delta (1-2\eps)\cdot 2\eps\cdot \frac N{M_{\rm reg}}
    \overset{\eqref{Rlarge}} \leq 4\Delta (1-\eps)^2\cdot 2\eps\cdot\frac N{t}\\
    &\overset{\eqref{sizeVi}}\leq 4\Delta\cdot 2\eps\cdot(1-\eps)\min\big\{|V_1|, |V_\ell|\big\}
  \leq 4\Delta\cdot 2\eps\cdot \min\big\{|V'_1|,|V_\ell'|\big\}\notag\\[6pt]
 &=2\Delta\cdot 2\eps\cdot\lambda.\notag
  \end{align}
  
So Lemma~\ref{blow-up} guarantees we can embedd $H$ as planned.  We will now embed the rest of $C_n\cup D$, that is, the long connections between the vertices of $H$.

 \begin{figure}[htb]
  \centering
  \subfloat[First, the graph $H$ is embedded in the pair $(V_1,V_\ell)$ via Lemma~\ref{blow-up}.]{\label{fig:caseAcycle1}
  \def\svgwidth{140pt}
  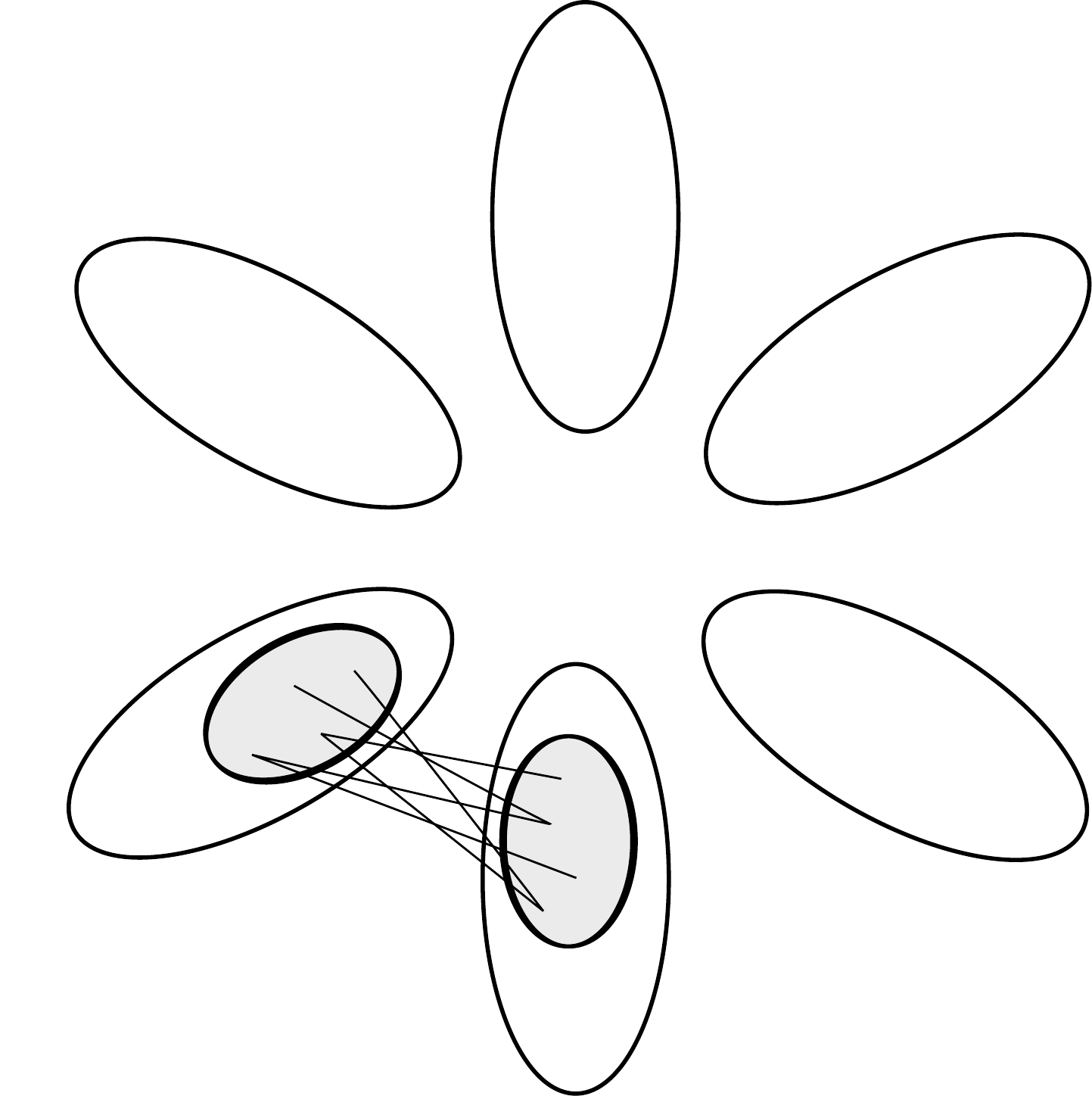}  
  \hfill
  \subfloat[Then we use Lemma~\ref{lem:path} to embed the long connections.]{\label{fig:caseAcycle2}  \def\svgwidth{160pt}
  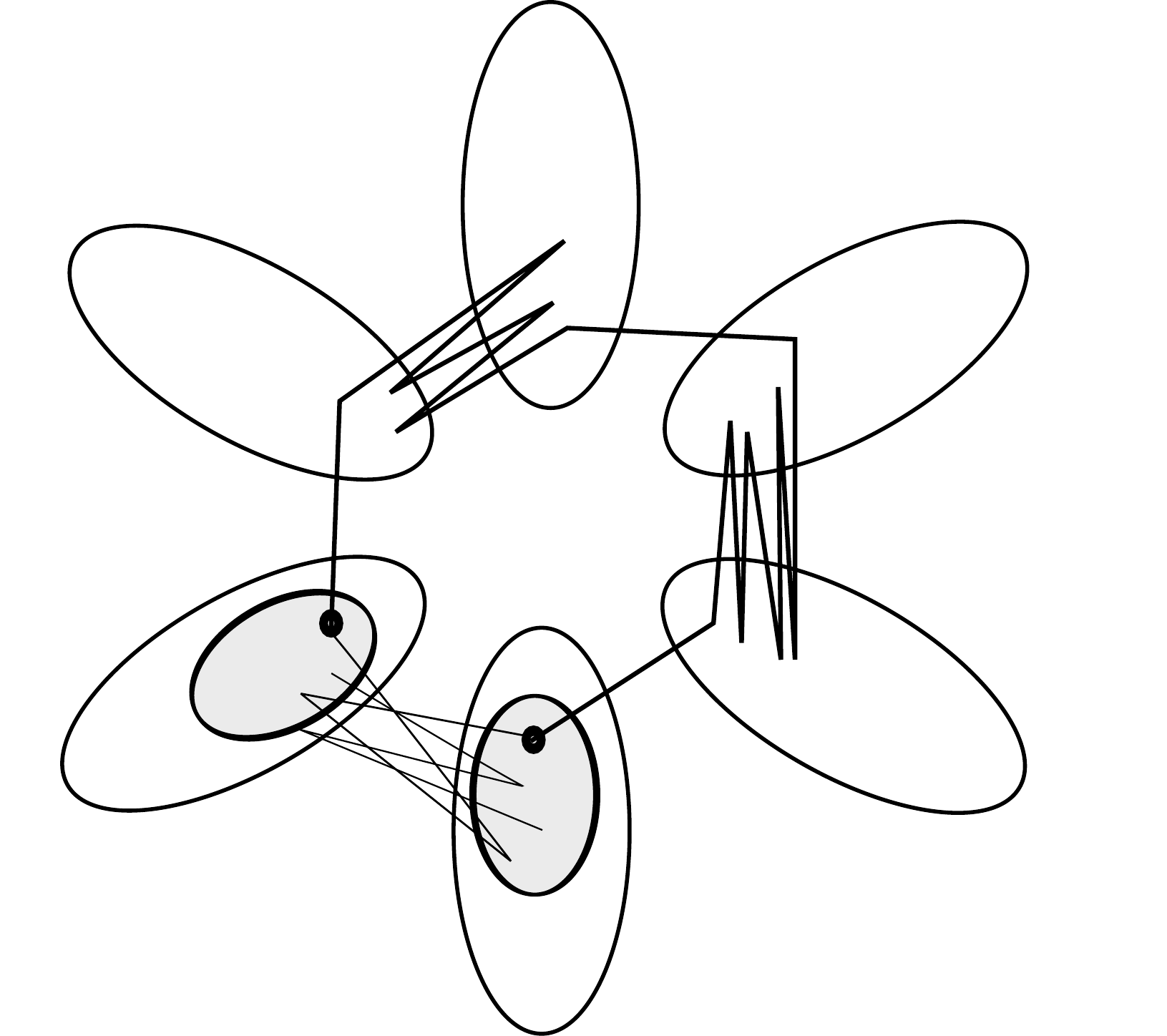}                  
  \caption{Our embedding of $C_n\cup D=H\cup\bigcup P_i$ in $G^r$.}
  \label{fig:caseAcycle}
\end{figure}
 \medskip
 
By~\eqref{U1U2paths}, we know that the edges of $C_n-E(H)$ form internally vertex-disjoint paths $P_1,P_2,\ldots, P_k$ such that for all $i\in[k]$ the first vertex   of $P_i$ is mapped to $a_i\in V_1$, while the last vertex of $P_i$ is embedded in  $b_i\in V_\ell$. By the choice of $V'_1$ and $V'_\ell$, we know that, for each $i\in[k]=\{1,\ldots k\}$,
\begin{align}\label{aibiepsreg}
&\text{$a_i$ is $\eps$-typical with respect to $V_2$,}\notag \\ \text{  and }&\text{$b_i$ is $\eps$-typical with respect to $V_{\ell -1}$.}
\end{align}
  
We now wish to employ Lemma~\ref{lem:path} to embed the remaining part of $H$ into $G^r$. We have just verified that~\eqref{lem:pathc} and~\eqref{lem:pathd} of Lemma~\ref{lem:path} hold. Also,~\eqref{lem:pathh} and ~\eqref{lem:pathf} are clear by construction.

For~\eqref{lem:pathg}, note that by construction,
$$
 k\leq |H|\overset{\eqref{sizeH}}< \eps^{1/2}|V_j|.
$$

 For~\eqref{lem:patha}, it suffices to observe that for each $i\in[k]$, 
 inequality \eqref{howlong} implies that
 $$
   p^i \geq z-3 \geq 8M_{\rm reg} \geq 3\ell.
  $$
  
 Let us next establish~\eqref{lem:pathb}. For this, it is enough to see that, as $M\geq n\geq 1/\eps$, we have
 \begin{align*}
  \sum_{i=1}^k p^i & \leq n\\
  &= \frac 23 (N+1)\\[6pt]
  &\leq \frac 23\cdot \frac{N}{1-\eps}\\[6pt]
  & \overset{\eqref{sizeVi}}\leq\frac 23\cdot \frac{|V_1|t}{(1-\eps)^2}\\[6pt]
  & \leq\frac 23 \cdot\frac{t}{1-3\eps^{1/4}} (1-2\epsilon^{1/4})|V_1| \\[6pt]
  & \leq\left(\frac23(1+\eps^{1/5})t-2\right)(1-2\epsilon^{1/4})|V_1|\\[6pt]
  & \overset{\eqref{evencycle}}\leq (\ell -2)(1-2\epsilon^{1/4})|V_1|.
 \end{align*}

\medskip
Finally, for~\eqref{lem:pathe} note that
\begin{align*}
 4\epsilon^{1/2}|V_j| 
&\overset{\eqref{sizeVi}}\geq 4\epsilon^{1/2}(1-\eps)\frac Nt \\[6pt]
&\geq 4\epsilon^{1/2}(1-\eps)\frac {n_0}t\\[6pt]
&\overset{\eqref{c}}\geq\frac{M_{\rm reg}n_{\ref{benevides}}}t \\[6pt]
&\overset{\eqref{Rlarge}}\geq n_{\ref{benevides}}.
\end{align*}
  
  This means we can use Lemma~\ref{lem:path} as planned and thus finish our embedding in this case.

 \subsubsection{Case~\eqref{evenstructure}}\label{sec:evenstructure}
 
We assume the colour $i$ is red and colour $3-i$ to be blue. Let $A_G$ denote the union of all vertices that belong to clusters in $A$. We estimate the number of blue edges in $A_G$ as follows: there are
\begin{itemize}
 \item at most $|A|\cdot\binom{n/t}{2}$ blue edges within clusters in $A$;
 \item at most $|A|\cdot\eps t\cdot (n/t)^2$ blue edges in irregular pairs in $A$;
 \item at most $\xi t |A|\cdot (n/t)^2$ blue edges in pairs corresponding to blue edges in $A$;
\item at most $\binom{|A|}{2}\cdot d (n/t)^2$ blue edges in the remaining pairs in $A$ (those which are coloured only by red).
\end{itemize}
Using the facts that $|A|\leq t$, that $1/t<\eps$, and that $\xi=\eps^{1/100}\leq d/4$, we find that $A_G$ contains in total at most 
$$
 \left(\frac \eps2+\eps+\xi+\frac d2\right)n^2
  \leq 
 d n^2
$$ 
blue edges. Since the size of $A_G$ is bounded by
$$
 |A_G|\geq |A|\cdot \frac{(1-\eps)N}{t}\overset{(B3)}\geq \left(1 -\xi\right)(1-\eps)n,
$$
we obtain that $A_G$ contains a subset $A'$ of size at least $|A_G|-2\sqrt{d}\geq (1-3d^{1/2})n$ such that, in $G^r$, each vertex of $A'$ has degree at least $|A'|-d^{1/2}n$.

Let $C$ be the set of all vertices in $V(G)\setminus A'$ that send at least $6\Delta d^{1/2} n$ red edges to $A'$ and suppose that $|A'|+|C|\geq n$. Since
$$
 2d^{1/2}\overset{\eqref{defd}}< \frac{2}{\Delta^{20}}<\frac{1}{\Delta^2+4},
$$ 
we may use Lemma \ref {GenEmbed} and conclude that the red graph $G^r[A'\cup C]$ contains $C_n\cup D$.
 
 So let us from now on assume that  $|A'|+|C|< n$. Setting $B':=V(G)\setminus (A'\cup C)$, this means that  $|B'|\geq n/2$. Moreover, by definition of $B'$, each vertex from $B'$ sends at least $|A'|-6\Delta d^{1/2} n$ blue edges to $A'$.
 We embed $C_n\cup D$ in $G^b[A', B']$ as follows.
 
 Say $C_n\cup D$ has the bipartition classes $X$ and $Y$. We know that $|X|=|Y|=|C_n|/2$. Now, embed all of $X$ into $B'$ arbitrarily. Then, we embed the vertices in $Y$ successively in the following way: Say we wish to embed some vertex $y\in Y$. Vertex $y$ has at most $\Delta$ neighbours in $X$. The images of these neighbours in $B'$ have at least $|A'|-\Delta\cdot 6\Delta d^{1/2} n>n/2$ common blue neigbors in $A'$. Since we have used at most $|Y|-1<n/2$ vertices of $A'$ in earlier steps, there is a vertex in $A'$ we may use to embed $y$. 
  Thus we manage to embed all of $Y$ as desired in $A'$, which finishes the embedding of $C_n\cup D$.
  
 \subsection{Odd $n$}\label{sec:odd}
 
 In this case, we wish to apply Theorem~\ref{jmy-odd} with parameters $\gamma:=\beta=\eps^{1/5}$, $\bar m:=t/2$, and $\bar M:=t$ to the graph $G:=R$. From~\eqref{Rlarge} and~\eqref{defmreg}, we get that $n> \gamma^{-6}$, and, clearly, we have $\bar M\geq (3/2+ 14\gamma)\bar m$.
 Finally, the fact that $\delta (G)\geq  M-\beta  m\geq \bar M-\gamma \bar m$ follows from~\eqref{mindegR}. 
 
Thus Theorem~\ref{jmy-odd} yields one of the following substructures:
 \begin{enumerate}[\bf (A)]
  \addtocounter{enumi}{2}
  \item \label{oddcycle} $R$ has a monochromatic odd cycle $C_\ell$  with $\ell\geq (1+\eps^{1/5})t/2$, or
  \item \label{oddstructure} there is a colour $i\in\{1,2\}$, and sets $A,B\subseteq V(R)$ so that 
  \begin{enumerate}[(D1)]
   \item all edges in $R[A]$ and in $R[B]$ have only colour $i$,\label{d1}
   \item all edges between $A$ and $B$ have only colour $3-i$, and\label{d2}
   \item $|A|,|B|\geq (1-\eps^{1/5})t/2$.\label{d3}
 \end{enumerate}  
\end{enumerate}
 Depending on which case we obtained, we split our proof further into two cases.
  
 \subsubsection{Case~\eqref{oddcycle}}\label{sec:oddcycle}
 
 Suppose that $C_\ell$ is red and contains the vertices $V_1, V_2, \ldots, V_\ell, V_1$ (in this order). First of all, observe that since $\ell$ is odd, the set of all $2$-chords of $C_\ell$ span an odd cycle $C'_\ell$ of the same length in $R$. Now, we split our proof further depending on whether or not $C'_\ell$ contains a red $2$-chord.
  
 \paragraph{$C'_\ell$  has a red $2$-chord.}\label{sec:1oddcycle}
 
 Let $V_\ell V_2$ be the red $2$-chord of $C_\ell$. Our plan is to use the Lemma~\ref{blow-up} to embed $H$ into the $\eps$-regular triangle $V_\ell V_1V_2$ such that all vertices of $H$ that have neighbours in $G-H$ are mapped to vertices of $V_\ell$ or $V_2$ that are typical in red with respect to $V_{\ell -1}$  or to $V_3$, respectively. 
 
  For this, let  $V_\ell'\subseteq V_\ell$ and $V_2'\subseteq V_2$ consist of all vertices that are $\eps$-typical in $G^r$ with respect to $V_{\ell-1}$ or to $V_3$, respectively. Then, by Fact~\ref{fact1}, we have  
 $$
  \lambda:=\min\{|V'_2|,|V'_\ell|\} \geq (1-\eps)|V_2|,
 $$
 from the definition of $\eps$-regularity we have that the pairs $(V'_\ell, V'_2)$, $(V'_\ell, V_1)]$ and $(V_1, V'_2)$ are $2\eps$-regular, and~\eqref{Hsmallenough} from the even case implies that $ |H|\leq 2\Delta\cdot 2\eps\cdot\lambda$.
  Hence Lemma~\ref{blow-up} with parameters $\Delta$, $p:=d-\eps$, $\eps':=2\eps$, $\lambda$, $s:=|H|$ and graphs $R:=K_3$, $H$, $G:=G^r[V'_\ell\cup V_1\cup V'_2]$ yields an embedding of the tripartite graph $H$ into $G$. Recall that the tripartition classes of $H$ were $U_1$, $U_2$ and $U_3$. We map $U_1$ to $V_1$, $U_2$ to $V_\ell'$ and $U_3$ to $V_1$.
  
We will now embed the rest of $C_n\cup D$, that is, the long connections between the vertices of $H$.  We plan to use Lemma~\ref{lem:path} on $G^r$ with $\ell^{Lem~\ref{lem:path}}:=\ell-1$ and $V_j^{Lem~\ref{lem:path}}:=V_{j+1}$.
 
 By~\eqref{U1U2paths}, we know that the edges of $C_n-E(H)$ form internally vertex-disjoint odd paths $P_1,P_2,\ldots, P_k$ such that, for all $i\in[k]$, the first vertex of $P_i$ is mapped to some  $a_i\in V_2$, while the last vertex of $P_i$ is embedded to some $b_i\in V_\ell$. By the choice of $V'_2$ and $V'_\ell$, we know that, for each $i\in[k]$, we have that
\begin{equation}\label{aibiepsregodd}
 \text{ $a_i$ is $\eps$-typical to $V_3$ and $b_i$ is $\eps$-typical to $V_{\ell -1}$.}
\end{equation}
  
  This means that conditions~\eqref{lem:pathh}, \eqref{lem:pathf}, \eqref{lem:pathc} and~\eqref{lem:pathd} of Lemma~\ref{lem:path} are satisfied. Conditions~\eqref{lem:pathe},~\eqref{lem:pathg} and~\eqref{lem:patha} follow as in the even case.
  
  For~\eqref{lem:pathb}, we calculate similarly as in the even case:
   \begin{align*}
  \sum_{i=1}^k p^i & \leq n\\
  &\leq \frac 12 (N+1)\\  
    & \leq\left(\frac12(1+\eps^{1/5})t-2\right)(1-2\epsilon^{1/4})|V_1|\\ 
  & \overset{\eqref{oddcycle}}\leq (\ell -2)(1-2\epsilon^{1/4})|V_1|.
 \end{align*}

  Thus with the help of Lemma~\ref{lem:path} we may complete the embedding of $H$ to an embedding of $C_n\cup D$.

 \paragraph{All $2$-chords of $C'_\ell$ are blue.}\label{sec:2oddcycles}
 
 Let $V'_1\subseteq V_1$ be the set of all vertices in $V_1$ that are typical in red with respect to  $V_\ell$ and $V_2$, and typical in blue with respect to $V_{\ell-1}$ and $V_3$. Since $|V_1'|\geq |V_1|-4\epsilon |V_1|$ by Fact~\ref{fact1}, we have
 \begin{align*}
 |V_1'|&\geq (1-4\epsilon) |V_1|\geq (1-5\epsilon)\frac{N}{t} \geq \frac{n}{M_{\rm reg}}\\
 &\geq 2^{c_2\Delta(\log\Delta)^2}\cdot12 \frac{\eps^2n}{M_{\rm reg}}\geq 2^{c_2\Delta(\log\Delta)^2}|H|\geq R(H).
 \end{align*}
Hence, by Theorem \ref{linram}, we find a monochromatic copy of $H$ inside $V'_1$. 
 
 Now only have to embed the long connections (in $G^r$ if our copy of $H$ is red, and in $G^b$ if our copy of $H$ is blue). This will be done as before, either in $C_\ell$ or in $C_\ell'$, with the help of Lemma~\ref{lem:path}. Now the input for Lemma~\ref{lem:path} is $\ell^{{\rm Lem}~\ref{lem:path}}:=\ell+1$ and $V_j^{{\rm Lem}~\ref{lem:path}}:=V_{j}$ for $j\in[\ell]$ and $V_{\ell +1}^{{\rm Lem}~\ref{lem:path}}:=V_{1}$.

Conditions~\eqref{lem:pathe}--\eqref{lem:pathd} of  Lemma~\ref{lem:path}  are seen to hold as above, so we can finish our embedding as planned.

 \subsubsection{Case~\eqref{oddstructure}}\label{sec:oddstructure}

Assume again that colour $i$ is red and $3-i$ is blue.

Again, let $A_G$, $B_G$ denote the union of all vertices that belong to clusters in $A$, or in $B$, respectively. We have  that
$$
 |A_G|\geq |A|\cdot\frac{(1-\eps)N}t\geq (1-\eps)(1-\eps^{1/5})\frac{2n-1+k}{2}\geq (1-2\eps^{1/5})n. 
$$
Similarly, $|B_G|\geq  (1-2\eps^{1/5})n$.

Now we shall estimate the number of blue edges in $G[A_G]$ and $G[B_G]$ and the number of red edges in $G[A_G, B_G]$. There are
\begin{itemize}
 \item at most $|A|\cdot\binom{n/t}{2}$ blue edges within the clusters in $A$;
 \item at most $|A|\cdot\eps t\cdot (n/t)^2$ blue edges in irregular pairs in $A$ (i.e.~non-edges of $R[A]$, here we use  \eqref{delta});
\item at most $\binom{|A|}{2}\cdot d (n/t)^2$ blue edges in the remaining pairs in $A$ (here we use (D1)).
\end{itemize}
Using the facts that $|A|\leq t$, that $1/t<\eps$, and that $\eps<d/4$, we find that $A_G$ contains in total at most
$$
 \left(\frac \eps2+\eps+\frac{d}{2}\right)n^2
  \leq 
 d n^2
$$ 
blue edges. In the same way we obtain that $B_G$ contains at most $dn^2$ blue edges.

As for the red edges in $G[A_G, B_G]$: There are 
\begin{itemize}
 \item at most $|A|\cdot\eps t\cdot (n/t)^2$ red edges in irregular pairs corresponding to non-edges in $R[A,B]$, and
\item at most $|A|\cdot |B|\cdot d (n/t)^2$ red edges in the pairs corresponding to edges in $R[A,B]$ (here we use (D2)).
\end{itemize}
Using the facts that $|A|,|B|\leq t$, that $1/\eps<t$, and that $\eps<d/4$, we find that $G[A_G,B_G]$ contains at most 
$$
 (\eps+d)n^2
  \leq 
 2d n^2
$$ 
red edges.  We remove at most $3\sqrt{d} n$ vertices with degree at least $\sqrt{d}n$ in one of $G^b[A']$, $G^b[B']$, and $G^r[A', B']$, and obtain sets $A'\subset A_G$ and $B'\subset B_G$ such that 
$$
 |A'|, |B'|\geq (1-2\eps^{1/5}-3\sqrt{d})n
$$
and such that the graphs $G^b[A']$, $G^b[B']$, and $G^r[A', B']$ have maximum degree at most $\sqrt{d}n$. 

Set $\eps'=2\eps^{1/5}+3\sqrt{d}$, and let $C_A$ and $C_B$ be the sets of all vertices in $V(G)\setminus (A'\cup B')$ that send at least $3\Delta\eps' n$ red edges to $A'$ or to $B'$, respectively. If $|C_A|+|A'|\geq n$, then we  use Lemma \ref {GenEmbed} (note that $\eps'<\frac 1{\Delta^2+4}$ by~\eqref{epsdef} and by~\eqref{defd}) and conclude that the red graph $G^r[A'\cup C_A]$ contains $C_n\cup D$.

So let us from now on assume that  $|C_A|+|A'|, |C_B|+|B'|< n$. This means that there is a set $S$ of $$k\ \leq \ 2n-1+k-1-(|C_A|+|A'| + |C_B|+|B'|)$$
 vertices which each send at least $|A'|-3\Delta\eps' n$ blue edges to $A'$ and at least $|B'|-3\Delta\eps' n$ blue edges to $B'$. Let  $A^S$ and $B^S$ be the set of their common blue neighbours in $A'$, or in $B'$, respectively. Note that
\begin{align}
|A^S|, |B^S|\ \geq \ (1-\eps')n-k\cdot 3\Delta\eps' n\  &=\ (1-(3k\Delta +1)\eps')n\label{ASBSbigenougH}\\ 
&\geq \ n/2,\label{ASBSbigenough}
\end{align}
where the last inequality follows from~\eqref{epsdef}.

 We embed $C_n\cup D$ as follows: let $Z$ be the set of $k$ independent vertices such that $C_n\cup D-Z$ is a bipartite graph with bipartition $X$ and $Y$. Clearly, the sets $X$ and $Y$ differ in size by at most $k$, and therefore we have that $|X|, |Y|\leq n/2$.
 
Embed all of $Z$ into $S$ and all of $X$ into $B^S$ arbitrarily (that $X$ fits into $B^S$ in ensured by~\eqref{ASBSbigenough}). Then embed the vertices in $Y$ successively, similarly as in the even case. For each vertex $y\in Y$ that we wish to embed in $A^S$, we know that the images of its at most $\Delta$ neighbours have at least 
\[
|A^S|-\Delta\cdot \sqrt d n \overset{\eqref{ASBSbigenougH}}\geq (1-4k\Delta \eps' -\sqrt d\Delta)n >0
\]
common neighbours in $A^S$ (the last inequality follows from~\eqref{epsdef} and from~\eqref{defd}), and thus there is space for embedding $y$ properly.

\end{document}

%% file: pathlemma.pdf_tex

\begingroup
  \makeatletter
  \providecommand\color[2][]{%
    \errmessage{(Inkscape) Color is used for the text in Inkscape, but the package 'color.sty' is not loaded}
    \renewcommand\color[2][]{}%
  }
  \providecommand\transparent[1]{%
    \errmessage{(Inkscape) Transparency is used (non-zero) for the text in Inkscape, but the package 'transparent.sty' is not loaded}
    \renewcommand\transparent[1]{}%
  }
  \providecommand\rotatebox[2]{#2}
  \ifx\svgwidth\undefined
    \setlength{\unitlength}{413.48980903pt}
  \else
    \setlength{\unitlength}{\svgwidth}
  \fi
  \global\let\svgwidth\undefined
  \makeatother
  \begin{picture}(1,0.33625388)%
    \put(0,0){\includegraphics[width=\unitlength]{pathlemma.pdf}}%
    \put(0.03460815,0.2554698){\color[rgb]{0,0,0}\makebox(0,0)[lb]{\smash{$a_1$}}}%
    \put(0.90814465,0.24499464){\color[rgb]{0,0,0}\makebox(0,0)[lb]{\smash{$b_1$}}}%
    \put(0.11915384,0.31160115){\color[rgb]{0,0,0}\makebox(0,0)[lb]{\smash{$q^1_1$}}}%
    \put(0.27925901,0.31273456){\color[rgb]{0,0,0}\makebox(0,0)[lb]{\smash{$q^1_2$}}}%
    \put(0.46187695,0.28857862){\color[rgb]{0,0,0}\makebox(0,0)[lb]{\smash{$q^1_3$}}}%
    \put(0.63265171,0.31547065){\color[rgb]{0,0,0}\makebox(0,0)[lb]{\smash{$q^1_4$}}}%
    \put(0.81300304,0.29612314){\color[rgb]{0,0,0}\makebox(0,0)[lb]{\smash{$q^1_5$}}}%
  \end{picture}%
\endgroup

%% file: caseAcycle1.pdf_tex

\begingroup
  \makeatletter
  \providecommand\color[2][]{%
    \errmessage{(Inkscape) Color is used for the text in Inkscape, but the package 'color.sty' is not loaded}
    \renewcommand\color[2][]{}%
  }
  \providecommand\transparent[1]{%
    \errmessage{(Inkscape) Transparency is used (non-zero) for the text in Inkscape, but the package 'transparent.sty' is not loaded}
    \renewcommand\transparent[1]{}%
  }
  \providecommand\rotatebox[2]{#2}
  \ifx\svgwidth\undefined
    \setlength{\unitlength}{415.37741052pt}
  \else
    \setlength{\unitlength}{\svgwidth}
  \fi
  \global\let\svgwidth\undefined
  \makeatother
  \begin{picture}(1,1.00369135)%
    \put(0,0){\includegraphics[width=\unitlength]{caseAcycle1.pdf}}%
    \put(0.27079289,0.15432498){\color[rgb]{0,0,0}\makebox(0,0)[lb]{\smash{$H$}}}%
    \put(0.63094729,0.03876742){\color[rgb]{0,0,0}\makebox(0,0)[lb]{\smash{$V_1$}}}%
    \put(-0.00654525,0.34306899){\color[rgb]{0,0,0}\makebox(0,0)[lb]{\smash{$V_6$}}}%
  \end{picture}%
\endgroup

%% file: caseAcycle2.pdf_tex

\begingroup
  \makeatletter
  \providecommand\color[2][]{%
    \errmessage{(Inkscape) Color is used for the text in Inkscape, but the package 'color.sty' is not loaded}
    \renewcommand\color[2][]{}%
  }
  \providecommand\transparent[1]{%
    \errmessage{(Inkscape) Transparency is used (non-zero) for the text in Inkscape, but the package 'transparent.sty' is not loaded}
    \renewcommand\transparent[1]{}%
  }
  \providecommand\rotatebox[2]{#2}
  \ifx\svgwidth\undefined
    \setlength{\unitlength}{469.47772675pt}
  \else
    \setlength{\unitlength}{\svgwidth}
  \fi
  \global\let\svgwidth\undefined
  \makeatother
  \begin{picture}(1,0.88803087)%
    \put(0,0){\includegraphics[width=\unitlength]{caseAcycle2.pdf}}%
    \put(0.26035436,0.13751502){\color[rgb]{0,0,0}\makebox(0,0)[lb]{\smash{$H$}}}%
    \put(0.55101168,0.02115473){\color[rgb]{0,0,0}\makebox(0,0)[lb]{\smash{$V_1$}}}%
    \put(0.72787247,0.43470198){\color[rgb]{0,0,0}\makebox(0,0)[lb]{\smash{$q^1_2$}}}%
    \put(0.55529051,0.21216891){\color[rgb]{0,0,0}\makebox(0,0)[lb]{\smash{$q^1_1$}}}%
    \put(0.44825896,0.19853674){\color[rgb]{0,0,0}\makebox(0,0)[lb]{\smash{$a_1$}}}%
    \put(0.20948318,0.30946078){\color[rgb]{0,0,0}\makebox(0,0)[lb]{\smash{$b_1$}}}%
    \put(-0.00579101,0.31379603){\color[rgb]{0,0,0}\makebox(0,0)[lb]{\smash{$V_6$}}}%
    \put(0.72787247,0.43470198){\color[rgb]{0,0,0}\makebox(0,0)[lb]{\smash{$q^1_2$}}}%
    \put(0.56592811,0.66463594){\color[rgb]{0,0,0}\makebox(0,0)[lb]{\smash{$q^1_3$}}}%
    \put(0.28879541,0.65740639){\color[rgb]{0,0,0}\makebox(0,0)[lb]{\smash{$q^1_4$}}}%
    \put(0.16847417,0.42705894){\color[rgb]{0,0,0}\makebox(0,0)[lb]{\smash{$q^1_5$}}}%
  \end{picture}%
\endgroup